\documentclass[12pt,letterpaper]{article}
\usepackage[utf8]{inputenc}
\usepackage{color}
\usepackage{amsmath}
\usepackage{amsthm}
\usepackage{amssymb}
\PassOptionsToPackage{normalem}{ulem}
\usepackage{ulem}
\usepackage[unicode=true,pdfusetitle,
 bookmarks=true,bookmarksnumbered=false,bookmarksopen=false,
 breaklinks=false,pdfborder={0 0 1},backref=page,colorlinks=true]
 {hyperref}
\hypersetup{
 linkcolor=blue, citecolor=blue}

\makeatletter

\pdfpageheight\paperheight
\pdfpagewidth\paperwidth

\numberwithin{equation}{section}
\numberwithin{figure}{section}
\theoremstyle{plain}
\newtheorem{thm}{\protect\theoremname}[section]
\theoremstyle{definition}
\newtheorem{defn}[thm]{\protect\definitionname}
\theoremstyle{remark}
\newtheorem{rem}[thm]{\protect\remarkname}
\theoremstyle{plain}
\newtheorem{lem}[thm]{\protect\lemmaname}
\theoremstyle{plain}
\newtheorem{cor}[thm]{\protect\corollaryname}
\theoremstyle{plain}
\newtheorem{prop}[thm]{\protect\propositionname}
\theoremstyle{plain}
\newtheorem{question}[thm]{\protect\questionname}

\@ifundefined{date}{}{\date{}}
\usepackage{graphicx}
\usepackage{appendix}

\usepackage{enumitem}
\usepackage[backref=page]{hyperref}
\usepackage{mathabx}

\makeatother

\providecommand{\corollaryname}{Corollary}
\providecommand{\definitionname}{Definition}
\providecommand{\lemmaname}{Lemma}
\providecommand{\propositionname}{Proposition}
\providecommand{\questionname}{Question}
\providecommand{\remarkname}{Remark}
\providecommand{\theoremname}{Theorem}

\begin{document}
\global\long\def\F{\mathcal{F} }%
\global\long\def\Aut{\mathrm{Aut}}%
\global\long\def\C{\mathbb{C}}%
\global\long\def\H{\mathcal{H}}%
\global\long\def\U{\mathcal{U}}%
\global\long\def\P{\mathcal{P}}%
\global\long\def\ext{\mathrm{ext}}%
\global\long\def\hull{\mathrm{hull}}%
\global\long\def\triv{\mathrm{triv}}%
\global\long\def\Hom{\mathrm{Hom}}%

\global\long\def\trace{\mathrm{tr}}%
\global\long\def\End{\mathrm{End}}%

\global\long\def\L{\mathcal{L}}%
\global\long\def\W{\mathcal{W}}%
\global\long\def\E{\mathbb{E}}%
\global\long\def\SL{\mathrm{SL}}%
\global\long\def\R{\mathbb{R}}%
\global\long\def\Z{\mathbf{Z}}%
\global\long\def\rs{\to}%
\global\long\def\A{\mathcal{A}}%
\global\long\def\a{\mathbf{a}}%
\global\long\def\rsa{\rightsquigarrow}%
\global\long\def\D{\mathbf{D}}%
\global\long\def\b{\mathbf{b}}%
\global\long\def\df{\mathrm{def}}%
\global\long\def\eqdf{\stackrel{\df}{=}}%
\global\long\def\ZZ{\mathcal{Z}}%
\global\long\def\Tr{\mathrm{Tr}}%
\global\long\def\N{\mathbb{N}}%
\global\long\def\std{\mathrm{std}}%
\global\long\def\HS{\mathrm{H.S.}}%
\global\long\def\e{\varepsilon}%
\global\long\def\c{\mathbf{c}}%
\global\long\def\d{\mathbf{d}}%
\global\long\def\AA{\mathbf{A}}%
\global\long\def\BB{\mathbf{B}}%
\global\long\def\u{\mathbf{u}}%
\global\long\def\v{\mathbf{v}}%
\global\long\def\spec{\mathrm{spec}}%
\global\long\def\Ind{\mathrm{Ind}}%
\global\long\def\half{\frac{1}{2}}%
\global\long\def\Re{\mathrm{Re}}%
\global\long\def\Im{\mathrm{Im}}%
\global\long\def\p{\mathfrak{p}}%
\global\long\def\j{\mathbf{j}}%
\global\long\def\uB{\underline{B}}%
\global\long\def\tr{\mathrm{tr}}%
\global\long\def\rank{\mathrm{rank}}%
\global\long\def\hh{\mathcal{H}}%
\global\long\def\h{\mathfrak{h}}%

\global\long\def\EE{\mathcal{E}}%
\global\long\def\PSL{\mathrm{PSL}}%
\global\long\def\G{\mathcal{G}}%
\global\long\def\Int{\mathrm{Int}}%
\global\long\def\acc{\mathrm{acc}}%
\global\long\def\awl{\mathsf{awl}}%
\global\long\def\even{\mathrm{even}}%
\global\long\def\z{\mathbf{z}}%
\global\long\def\id{\mathrm{id}}%
\global\long\def\CC{\mathcal{C}}%
\global\long\def\cusp{\mathrm{cusp}}%
\global\long\def\new{\mathrm{new}}%

\global\long\def\LL{\mathbb{L}}%
\global\long\def\M{\mathbf{M}}%
\global\long\def\I{\mathcal{I}}%
\global\long\def\X{X}%
\global\long\def\free{\mathbf{F}}%
\global\long\def\into{\hookrightarrow}%
\global\long\def\Ext{\mathrm{Ext}}%
\global\long\def\B{\mathcal{B}}%
\global\long\def\Id{\mathrm{Id}}%
\global\long\def\Q{\mathbb{Q}}%

\global\long\def\O{\mathcal{T}}%
\global\long\def\Mat{\mathrm{Mat}}%
\global\long\def\NN{\mathrm{NN}}%
\global\long\def\nn{\mathfrak{nn}}%
\global\long\def\Tr{\mathrm{Tr}}%
\global\long\def\SGRM{\mathsf{SGRM}}%
\global\long\def\m{\mathbf{m}}%
\global\long\def\n{\mathbf{n}}%
\global\long\def\k{\mathbf{k}}%
\global\long\def\GRM{\mathsf{GRM}}%
\global\long\def\vac{\mathrm{vac}}%
\global\long\def\SS{\mathcal{S}}%
\global\long\def\red{\mathrm{red}}%
\global\long\def\V{V}%
\global\long\def\SO{\mathrm{SO}}%
\global\long\def\Gd{\Gamma^{\vee}}%
\global\long\def\fd{\mathrm{fd}}%
\global\long\def\perm{\mathrm{perm}}%
\global\long\def\tos{\xrightarrow{\mathrm{strong}}}%
\global\long\def\Mat{\mathbf{Mat}}%
\global\long\def\Perm{\mathbf{Perm}}%
\global\long\def\PMF{\mathrm{P}\Mat\mathrm{F}}%
\global\long\def\PPF{\mathrm{P}\Perm\mathrm{F}}%
\global\long\def\KK{\mathcal{K}}%

\vspace{-3in} 
\title{Strong convergence of \\
unitary and permutation\\
 representations of discrete groups}
\author{Michael Magee}
\maketitle
\begin{abstract}
We survey a research program on the strong convergence of unitary
and permutation representations of discrete groups. We also take the
opportunity to flesh out details that have not appeared elsewhere. 

\tableofcontents{}
\end{abstract}

\section{Definitions}

Throughout the article $\Gamma$ denotes an infinite discrete group.
For $n\in\N$ we write $\U(n)$ for the group of complex $n\times n$
unitary matrices and $S_{n}$ for the group of permutations of $[n]\eqdf\{1,\ldots,n\}$.
\begin{defn}[Strong convergence, $\PMF$]
If $\{\rho_{i}:\Gamma\to\U(n_{i})\}_{i=1}^{\infty}$ are a sequence
of (possibly random) finite dimensional unitary representations of
$\Gamma$, say $\rho_{i}$ strongly converge to the regular representation
(almost surely, or in probability, if $\rho_{i}$ are random) if for
any $z\in\C[\Gamma]$,
\[
\lim_{i\to\infty}\|\rho_{i}(z)\|=\|\lambda_{\Gamma}(z)\|
\]
(a.s., or in probability, respectively) where $\lambda_{\Gamma}:\Gamma\to U(\ell^{2}(\Gamma))$
is the left regular representation. The norms above are operator norms.
We write $\rho_{i}\tos\lambda_{\Gamma}$ in this event. If $\Gamma$
has such a sequence of unitary representations then we say $\Gamma$
is purely matricial field ($\PMF$).
\end{defn}

\begin{rem}
~
\begin{enumerate}
\item \uline{Weak convergence.} Some authors ask for weak convergence as
part of the definition of strong convergence, but we do not. Weak
convergence is the statement that for all $z\in\C[\Gamma]$, 
\[
\lim_{i\to\infty}\tr[\rho_{i}(z)]=\tau(z)
\]
where $\tr=n_{i}^{-1}\Tr$ is the normalized trace on $n_{i}\times n_{i}$
matrices and $\tau(z)$ is the canonical tracial state on the reduced
$C^{*}$-algebra $C_{\red}^{*}(\Gamma)$. If $C_{\red}^{*}(\Gamma)$
has a unique tracial state then strong convergence implies weak convergence
(e.g. \cite[Lemma 6.1]{louder2023strongly})\footnote{We first heard this observation from Beno\^{i}t Collins in June 2022.}.
We now know exactly when $C_{\red}^{*}(\Gamma)$ has a unique tracial
state by Breuillard-Kalantar-Kennedy-Ozawa \cite[Thm. 1.6]{BKKO}:
this is so when $\Gamma$ has no non-trivial normal amenable subgroup.
\item \uline{Nomenclature.} Let~$M_{n_{i}}$denote the complex matrix $C^{*}$-algebra
of dimension $n_{i}$, $\ell^{\infty}(\prod_{i}M_{n_{i}})$ denote
the bounded sequences in the product, and $\I$ the closed two sided
ideal of sequences that converge to zero. If $\rho_{i}\tos\lambda_{\Gamma}$
then $\prod_{i}\rho_{i}$ descends (and extends) to an embedding
\[
C_{\red}^{*}(\Gamma)\hookrightarrow\ell^{\infty}(\prod_{i}M_{n_{i}})/\I.
\]
If there exists any such embedding then $C_{\red}^{*}(\Gamma)$ is
called matricial field by Blackadar and Kirchberg in \cite{BlackadarKirchberg2}.
But conversely such an embedding does not (a priori) have to factor
through $\ell^{\infty}(\prod_{i}M_{n_{i}})$ when restricted to $\C[\Gamma]$.
Hence the adjective `purely'. The concept of purely matricial field
was introduced in Magee--de la Salle \cite{MageeSalle}.
\item More generally, if $G$ is a locally compact topological group, we
can extend our definition by replacing $\C[\Gamma]$ by the continuous
compactly supported complex functions on $G$.
\end{enumerate}
\end{rem}

A first attempt to require strong convergence to factor through permutations
would be that the $\text{\ensuremath{\rho_{i}}}$ are the composition
of some $\phi_{i}\in\Hom(\Gamma,S_{n})$ with the ($n$-dimensional)
permutation representations of $S_{n}$. However, since this gives
rise to $\rho_{i}$ with non-zero invariant vectors, this can never
work if $\Gamma$ is non-amenable, which is not satisfactory. So the
definition is modified in the following way. Let $\std$ denote the
$(n-1)$-dimensional irreducible subrepresentation of the defining
representation of $S_{n}$. 
\begin{defn}[$\PPF$]
If there exist a sequence of homomorphisms $\{\phi_{i}:\Gamma\to S_{n_{i}}\}_{i=1}^{\infty}$
such that 
\[
\{\rho_{i}\eqdf\std\circ\phi_{i}\}_{i=1}^{\infty}
\]
strongly converge to the regular representation, then we say $\Gamma$
is purely permutation field ($\PPF$).
\end{defn}

Strong convergence, and the associated above properties $\PMF$ and
$\PPF$ have recently found powerful applications in a surprising
range of settings including but not limited to the spectral geometry
of graphs \cite{BordenaveCollins} and hyperbolic manifolds \cite{HideMagee,louder2023strongly,MageeThomas},
the theory of minimal surfaces (observed by Song \cite{Song2024,Song2024a},
and the Peterson-Thom conjecture \cite{Peterson2011} (as observed
by Hayes \cite{Hayes2022}) on subalgebras of free group factors.

\subsection*{Acknowledgments}

This paper is a survey of some of my results with collaborators together
with things I thought about during a stay at I.A.S. during 2023-2024.
I am extremely grateful to the I.A.S. for this opportunity. Some particular
elements of this year have had a large effect on this document. The
conversations I had with Mikael de la Salle who patiently explained
many things to me. The seminar at Princeton University organized by
Peter Sarnak on `$C^{*}$-algebras and related topics' was extremely
stimulating and sparked many conversations. Indeed, I also learned
many interesting things from subsequent conversations with Ramon van
Handel and his lectures in the seminar.

\uline{Funding}: This material is based upon work supported by the
National Science Foundation under Grant No. DMS-1926686. This project
has received funding from the European Research Council (ERC) under
the European Union’s Horizon 2020 research and innovation programme
(grant agreement No 949143).

\section{$\protect\PMF$ and $\protect\PPF$}

We begin with some basic properties of $\PMF$ and $\PPF$. In this
section $\Lambda\leq\Gamma$ are discrete groups.
\begin{lem}[Restriction to subgroups]
I\label{lem:restriction}f $\Gamma$ is $\PMF$ (resp. $\PPF$) then
$\Lambda$ is $\PMF$ (resp. $\PPF$).
\end{lem}

\begin{proof}
Suppose $z\in\C[\Lambda]\leq\C[\Gamma]$. If $\rho_{i}\tos\lambda_{\Gamma}$
then $\|\rho_{i}(z)\|\to\|\lambda_{\Gamma}(z)\|$ as $i\to\infty$.
But as a $\C[\Lambda]$-module,
\[
\ell^{2}(\Gamma)\cong\bigoplus_{[\gamma]\in\Lambda\backslash\Gamma}\ell^{2}(\Lambda\gamma)\cong\bigoplus_{[\gamma]\in\Lambda\backslash\Gamma}\ell^{2}(\Gamma)
\]
 so $\|\lambda_{\Gamma}(z)\|=\|\lambda_{\Lambda}(z)\|$. Hence if
$\rho'_{i}$ is the restriction of $\rho_{i}$ to $\Lambda$, $\rho'_{i}\tos\lambda_{\Lambda}$.
If the representations factor through $S_{n_{i}}$ then their restrictions
still do.
\end{proof}
\begin{lem}[Induction to finite index overgroups]
\label{lem:Induction1}If $\Lambda$ is $\PMF$ and finite index
in $\Gamma$ then $\Gamma$ is also $\PMF$.
\end{lem}

Lemma \ref{lem:Induction1} is a special case of a more general phenomenon
(induction from co-compact lattices) that will be covered later in
the paper (see Theorem \ref{thm:induction}).

\subsection{Amenable groups. }

The following argument about amenable groups was obtained in conversations
with Mikael de la Salle. 

Say that $\Gamma$ is residually linear (RL) (resp. residually finite
(RF)) if it embeds into a product of $\mathrm{GL}_{n_{i}}$ (resp.
$S_{n_{i}})$. A theorem of Malcev \cite{Malcev} states that every
finitely generated (f.g.) linear group is residually finite. Hence
if $\Gamma$ is f.g. and RL then it is RF. 

If $\Gamma$ is not RL then there is $\gamma\in\Gamma$ that is killed
by every homomorphism to some $\mathrm{GL}_{n}$. Then considering
\[
z=\gamma-\id\in\C[\Gamma],
\]
we have 
\[
\|\rho(z)\|=0
\]
 for every $\rho:\Gamma\to\mathrm{GL}_{n}$. But $\|\lambda(z)\|$
is not zero (or else $\lambda(z)=0$ but $\C[\Gamma]$ always embeds
to $C_{\red}^{*}(\Gamma)$). The upshot is that:
\begin{lem}
\label{lem:not-PMF}A $\PMF$ group is residually linear. Hence it
is also either residually finite or not finitely generated.
\end{lem}

On the other hand, if $\Gamma$ is countable and RF then by projecting
to large enough factors of the $\prod_{i}S_{n_{i}}$ we obtain a sequence
of homomorphisms $\phi_{j}:\Gamma\to S_{N_{j}}$ such that each $\phi_{j}$
injects on a finite set $S_{j}$ such that $\cup_{j}S_{j}=\Gamma.$
For any $z\in\C[\Gamma]$ let $\Phi_{j}$ denote the unitary representation
of $\Gamma$ obtained by composing $\phi_{j}$ with the standard representation
of $S_{N_{j}}$.

Because we eventually inject on the support of $(zz^{*})^{p}$ we
have then
\[
\tr[\Phi_{j}(zz^{*})^{p}]=\tau((zz^{*})^{p})-\frac{1}{N_{j}}\epsilon((zz^{*})^{p})
\]
where $\epsilon$ is the state associated to the trivial representation
of $\Gamma$, for $j\gg_{p}1$. Since the trivial representation is
weakly contained in the regular representation of $\Gamma$ we have
\[
\epsilon((zz^{*})^{p})\leq\|\lambda(z)\|^{2p}.
\]
Hence we also have 
\[
\|\Phi_{j}(z)\|^{2p}=\|\Phi_{j}(zz^{*})^{p}\|\geq\tr[\Phi_{j}(zz^{*})^{p}]\to_{i\to\infty}\tau((zz^{*})^{p})=(\|\lambda(z)\|+o_{p\to\infty}(1))^{2p}
\]
 where the last equality used that the trace $\tau$ is faithful.
Taking $p$ large and fixed and letting $i\to\infty$ for each large
$p$ we obtain
\[
\liminf_{j\to\infty}\|\Phi_{j}(z)\|\geq\|\lambda(z)\|.
\]
But on the other hand for amenable groups $\|\lambda(z)\|\geq\|\Phi_{i}(z)\|$
for all $z\in\C[\Gamma]$ because all finite dimensional unitary representations
of amenable groups are weakly contained in the regular representation.
So in fact 
\[
\lim_{j\to\infty}\|\Phi_{j}(z)\|=\|\lambda(z)\|.
\]
Hence in summary \uline{for amenable groups}
\[
\mathrm{countable}+\mathrm{RF}\implies\PPF
\]
and 
\[
\PMF\implies\mathrm{RL},
\]
\[
\mathrm{f.g.}+\,\PMF\implies\mathrm{RF}.
\]
What is perhaps surprising is that \uline{all} amenable groups have
f.d. \emph{approximate representations} (in norm sense) that strongly
converge to the regular representation \cite{TikuisisWhiteWinter}.
This appears to leave open the question of whether non-countable discrete
amenable groups are $\PMF$ or $\PPF$.

\subsection{Free groups}

Let $\free_{r}$ denote a free non-abelian group of rank $r\geq2$.

The original interest in property $\PMF$ for non-amenable groups
comes from an observation of Voiculescu \cite{voic_quasidiag} that
establishing `$\free_{r}$ is $\PMF$' --- albeit not in this language
--- would settle the then outstanding problem in operator algebras
to prove the $K$-theoretic construct $\Ext(C_{\red}^{*}(\free_{r}))$
has non-invertible elements. 

Motivated by this application to $K$-theory, Haagerup and Thorbjørnsen
proved $\free_{r}$ is $\PMF$ in the main theorem of \cite{HaagerupThr}.
Importantly, this, and all subsequent proofs of this fact, rely on
random matrix theory. As a result, there is a lacuna in the field:
we do not know how to construct explicit f.d. representations of free
groups that strongly converge to the regular representation\footnote{This problem seems known in the community, it is certainly not my
own, and it was highlighted to me by Avi Wigderson.}
Haagerup and Thorbjørnsen used GUE random matrices and an application
of functional calculus to prove their result, and as a result, it
left open the question of whether Haar distributed random representations
of $\free_{r}$ a.s. strongly converge to the regular representation.
This was proved by Collins and Male in \cite{Collins2014}.

Later on, Bordenave and Collins \cite{BordenaveCollins} proved that
$\free_{r}$ is $\PPF$ by use of random permutation representations.
To explain their motivation and to explain the motivation for $\PPF$
in general, we detour to discuss Friedman's theorem.

Friedman's theorem \cite{Friedman}, formerly conjecture of Alon \cite{Alon},
states informally that random regular graphs are almost optimal expanders.
Alon did not state which random model to use, but besides, there are
contiguity results that tell us it does not matter for the purposes
of proving the above statement \cite{Wormald}. For simplicity of
exposition assume that the random regular graph has degree $2r$ and
$n$ vertices, with $r$ fixed and $n\to\infty$.

If $\sigma_{1},\ldots,\sigma_{r}$ are i.i.d. uniformly random elements
of $S_{n}$ then the Schreier graph for these generators and the action
of $S_{n}$ on $\{1,\ldots,n\}$ has adjacency matrix
\[
2r\cdot\id_{\C}\oplus\text{\ensuremath{\left[\std(\sigma_{1})+\std(\sigma_{1}^{-1})+\cdots+\std(\sigma_{r})+\std(\sigma_{r}^{-1})\right]}}.
\]
The $2r\cdot\id_{\C}$ factor corresponds to the trivial eigenvalue
$2r$ and in this setting, asymptotically optimal expansion amounts
to
\[
\|\std(\sigma_{1})+\std(\sigma_{1}^{-1})+\cdots+\std(\sigma_{r})+\std(\sigma_{r}^{-1})\|\to2\sqrt{2r-1}
\]
(in probability) where the right hand side is optimal by a result
of Alon--Boppana \cite{Nilli}. To put the above in a more symmetric
form, and to connect to our other discussions, if 
\[
\phi(x_{i})\eqdf\sigma_{i},\quad\rho(x_{i})=\std(\phi(x_{i}))=\std(\sigma_{i})
\]
 then $\rho_{i}$ is a random representation of the type appearing
in the definition of $\PPF$ and Friedman's theorem states
\[
\|\rho_{i}(x_{1}+x_{1}^{-1}+\cdots+x_{r}+x_{r}^{-1})\|\to\|\lambda_{\free_{r}}(x_{1}+x_{1}^{-1}+\cdots+x_{r}+x_{r}^{-1})\|=2\sqrt{2r-1}
\]
in probability. This is a very restricted version of $\PPF$ concerning
only the linear element
\[
z=x_{1}+x_{1}^{-1}+\cdots+x_{r}+x_{r}^{-1}\in\C[\free_{r}]
\]
 whereas $\PMF$ and $\PPF$ ask for the same for all(!) elements
of $\C[\Gamma]$. This is precisely what Bordenave--Collins prove
in \cite{BordenaveCollins}: for all $z\in\C[\free_{r}]$
\[
\|\rho_{i}(z)\|\to\|\lambda_{\free_{r}}(z)\|
\]
in probability. We cannot leave this discussion without remarking
that very recently, a beautiful and short proof of Friedman's theorem
has been obtained by Chen, Garza-Vargas, Tropp, and van Handel, \cite{chen2024new}.
They also extend the proof without adding much length to give a new
proof of Bordenave--Collins' theorem.

\subsection{Non-free examples}

\emph{Limit groups and surface groups. }Motivated in part by applications
to hyperbolic surfaces (see $\S$\ref{sec:Applications}), in \cite{louder2023strongly}
Louder and I studied \emph{limit groups: }finitely generated groups
such that for every finite subset, there is a homomorphism to a free
group that injects on that set.
\begin{thm}[Louder--Magee]
Every limit group is $\PPF$.
\end{thm}

The full strength of this result uses a result of Sela \cite{SelaI,cg}
stating that every limit group embeds into an iterated sequence of
(extensions of centralizers), beginning with a free group. This is
combined with a quantitative version of a lemma of Baumslag \cite{baumslag_on_generalised}
and Haagerup's inequality \cite{Haagerup}. The following is one of
the motivations.
\begin{cor}
\label{cor:surfaces}The fundamental group of an orientable surface
of genus at least two is $\PPF$.
\end{cor}

While stated as a corollary, Sela's work is not needed here as a genus
two surface group 
\[
\langle a,b,c,d\,|\,[a,b][c,d]\,\rangle
\]
explicitly embeds to the extension of centralizers
\[
\langle a,b,t\,|\,[t,[a,b]]\,\rangle
\]
via $a\mapsto a,b\mapsto b,c\mapsto b^{t},d\mapsto a^{t}$; the injectivity
of this map uses normal form for amalgamated products. Higher genus
surface groups embed into this example via covering spaces.

\emph{Right-angled Artin Groups and related examples. }A finitely
generated right-angled Artin group (RAAG)\emph{ }is generated by finitely
many generators with only relations that some subset of the pairs
of generators are commuting pairs. Besides interpolating between free
non-abelian and free abelian groups, they turn out to be fairly universal
in that many natural families of groups virtually embed into RAAGs.
That is, they embed after passing to a finite index subgroup. This
includes the following classes of groups.
\begin{enumerate}
\item Closed hyperbolic three manifold fundamental groups. (Bergeron--Wise
\cite{BWBoundary} and Agol -- \cite{Agol})
\item Non-compact finite volume three manifold fundamental groups\footnote{I thank Jean-Pierre Mutanguha for making me aware of this\@.}.
(Wise -- \cite[Thm. 14.29]{WiseBook})
\item Arithmetic `standard type' hyperbolic $n$ manifold groups with $n\geq4$.
(Bergeron--Haglund--Wise \cite{BHW})
\item Any Coxeter group. (Haglund--Wise \cite{HaglundWiseCoxeter})
\item Any one-relator group with torsion. (Wise -- \cite[Cor. 18.1]{WiseBook})
\item Any word-hyperbolic cubulated group. (Agol -- \cite{Agol})
\end{enumerate}
The above results also all rely on a result of Haglund and Wise \cite{HaglundWiseSpecial}
stating that fundamental groups of compact special cube complexes
embed into RAAGs.

By Lemmas \ref{lem:restriction} and \ref{lem:Induction1}, $\PMF$
passes from RAAGs to those groups above. The `master theorem' that
RAAGs are indeed P$\Mat$F was obtained in joint work with Thomas
\cite{MageeThomas}.
\begin{thm}[Magee--Thomas]
Finitely generated RAAGs are $\PMF$.
\end{thm}

It is not true however that in general RAAGs are P$\Perm$F. The following
proposition has not appeared elsewhere in print.
\begin{prop}
\label{prop:F23isnotppF}$\mathbf{F_{2}}\times\mathbf{F_{2}}\times\mathbf{F_{2}}$
is not $\PPF$.
\end{prop}

\begin{proof}
Let $G_{i}$ denote the copy of $\mathbf{F}_{2}$ embedded at the
$i$\textsuperscript{th} factor, $i=1,2,3.$ Let $\{\rho_{j}\}_{j=1}^{\infty}$
be a putative sequence of permutation representations of $\mathbf{F_{2}}\times\mathbf{F_{2}}\times\mathbf{F_{2}}$
such that $\std\circ\rho_{i}$ strongly converges to the regular representation.
We interpret $\rho_{j}$ as a linear representation by composition
with the defining representation. Let $H_{i,j}\eqdf\rho_{j}(G_{i})\leq S_{n_{j}}$.
Since $\mathbf{F}_{2}$ is non-amenable, each $H_{i,j}$ has at most
one dimensional space of invariant vectors in the defining representation.
This means that each $H_{i,j}$ must be a transitive subgroup of $S_{n_{j}}$
(acting transitively on $[n_{j}]$) for $j\gg1$. A theorem of finite
group theory\footnote{Indeed, $J_{1}$ must act freely on $[N]$, since $J_{2}$ preserves
the fixed points of any $j\in J_{1}$. This identifies $J_{1}$ with
$[N]$ via $j\mapsto j.1$ and makes the action of $J_{1}$ into the
left regular action. Now $J_{2}\leq\text{\ensuremath{\mathrm{Perms}(J_{1})} }$is
in the centralizer of the left regular action of $J_{1}$, which is
the right regular action of $J_{1}$ \cite[Thm. 6.3.1]{HALL}. But
since $J_{2}$ is transitive, it must be isomorphic to $J_{1}$ and
equal to the whole right regular action.} states that two commuting transitive permutation groups (say $J_{1}$
and $J_{2}$ in $S_{N})$ only arise in the following way: there is
some auxiliary group $U$ with $|U|=N$, and injective morphisms $\iota_{1}:J_{1}\cong\lambda(U)$
and $\iota_{2}:J_{2}\cong\rho(U)$ where $\lambda(U)$ (resp. $\rho(U)$)
is the permutation group induced by multiplication by $U$ on its
left (resp. right).

In particular, in the situation above, $H_{1,j}$ and $H_{2,j}$ can
be identified with $\lambda(U_{j})$ and $\rho(U_{j})$ for some $j$.
There are now different ways to conclude; we note that $H_{3,j}\cong U_{j}$
by the same observation but in the above model, it is $U_{j}$ acting
on itself by permutations and commuting with $\lambda(U_{j})$ and
$\rho(U_{j})$. Only the identity in $U_{j}$ can do this, so $U_{j}$
is the trivial one element group. Then $\std\circ\rho_{j}$ is the
0-dimensional representation for $j\gg1$. This is obviously a contradiction.
\end{proof}
As a byproduct of the above proof, one sees that any putative sequence
of permutation representations of $\mathbf{F_{2}}\times\mathbf{F_{2}}$
such that $\std\circ\rho_{i}$ strongly converges to the regular representation
has to have quite a particular structure. However, there is a candidate
that fits this structure.
\begin{question}
For $n\in\N$ let $\theta_{n}$ and $\phi_{n}$ denote uniformly random
permutation representations $\mathbf{F}_{2}\to S_{n}$. Then 
\[
\lambda\circ\theta_{n}:\mathbf{F}_{2}\to S_{n!}=\Perm(S_{n}),\,\rho\circ\phi_{n}:\mathbf{F}_{2}\to S_{n!}=\Perm(S_{n})
\]
are two commuting permutation representations of $\mathbf{F}_{2}$.
Is it true that as $n\to\infty$,
\[
\std_{n!}\circ[\lambda\circ\theta_{n}\times\rho\circ\phi_{n}]\tos\lambda_{\mathbf{F}_{2}\times\mathbf{F_{2}}}
\]
in probability?
\end{question}

Notice that if the above is true it entails that as $n\to\infty$,
$\lambda\circ\theta_{n}\tos\lambda_{\mathbf{F}_{2}}$ in probability.
This is far from known. In fact it is not known whether 
\begin{equation}
\|\lambda\circ\theta_{n}[x_{1}+x_{2}+x_{1}^{-1}+x_{2}^{-1}]\|\to2\sqrt{3}\label{eq:optimal-gap-cayley-sn}
\end{equation}
in probability; or the same with `2' replaced by any $d\geq0$ and
$\sqrt{3}$ replaced by $\sqrt{2d-1}$. This is even stronger than
the also open question:

\emph{`Are most fixed degree }\emph{\uline{Cayley}}\emph{ graphs of
$S_{n}$ uniform expanders?'}. ((\ref{eq:optimal-gap-cayley-sn})
suggests they are moreover almost optimal expanders)

We note here an important and perhaps relevant result of Kassabov
\cite{Kassabov} giving the \uline{existence} of fixed degree uniformly
expanding Cayley graphs of $S_{n}$ with $n\to\infty$. In this vein
there is a very recent work of E. Cassidy \cite{Cassidy2023a,Cassidy2024}
who proves the strong convergence in probability of 
\[
\rho_{n}:\free_{r}\to\U(N(n))
\]
\[
\rho_{n}(x_{i})\eqdf\pi_{\lambda}(\sigma_{i})
\]
where $\sigma_{i}$ are as above and $\pi_{\lambda}$ is the irreducible
representation of $S_{n}$ corresponding to a Young diagram $\lambda$
with $1\leq|\lambda|\leq n^{\frac{1}{13}}$ boxes. In fact his result
is uniform in this regime of $\lambda$. An analogous result for $U(n)$
(with slightly worse constant) was obtained recently with de la Salle
\cite{MageeSalle2}. Cassidy's result for $|\lambda|$ up to and including
$n$ would answer the above questions about Cayley graphs of $S_{n}$.

The fact that not all RAAGs are $\PPF$ does not have a direct implication
to e.g. closed hyperbolic 3 manifold groups in general because when
one induces a permutation representation from an infinite index subgroup
of a RAAG the resulting permutation representation is of an infinite
set and so cannot be used to potentially prove $\PPF$ for the RAAG.

\subsection{Non-examples\label{subsec:Non-examples}}

We discuss here f.g. residually finite groups\footnote{Non f.g. but residually linear groups have not really been considered
in the context of $\PMF$ to the author's knowledge and this might
be an interesting thing to investigate further.}. With M. de la Salle we established \cite{MageeSalle} that $\SL_{d}(\Z)$
is not $\PMF$ for $d\geq4$. This leaves a curious gap at $d=3$.
The reason for this gap is that we rely on the following fact, established
in \emph{(ibid.):} Every non-trivial finite dimensional unitary representation
of $\SL_{4}(\Z)$ has a non-zero $\SL_{2}(\Z)$-invariant vector.
This in turn means that the action of $\SL_{2}(\Z)$ in this representation
has no spectral gap, and so when restricted to $\SL_{2}(\Z)$, a putative
sequence of representations of $\SL_{4}(\Z)$ that strongly converge
to the regular representation, cannot converge to the regular representation
of $\SL_{2}(\Z)$ --- which does have a spectral gap. In light of
Lemma \ref{lem:restriction} this is a contradiction.

However, there are f.d. irreducible unitary representations of $\SL_{3}(\Z)$
with dimension tending to infinity and without non-zero $\SL_{2}(\Z)$-invariant
vectors. See \emph{(ibid.) }for details --- this example is due to
Deligne. 

\subsection{Connection to the Fell topology on the unitary dual}

Suppose $G$ is a locally compact topological group. We assume general
familiarity with the Fell topology on the unitary dual of $G$ (the
equivalence classes of continuous unitary representations of $G$)
and also the notion of weak containment of representations\footnote{In the Princeton seminar I proved Proposition \ref{prop:abstract-strong-convergence-consequence}
in the case the compact set $K$ was a point. M. de la Salle pointed
out the extension to compacts.}.
\begin{prop}
\label{prop:abstract-strong-convergence-consequence}Suppose that
$G$ is any locally compact group. Suppose that $\pi_{i}:G\to\U(\H_{i})$
are any sequence of representations of $\Gamma$ that strongly convergence
to a unitary representation $\pi_{\infty}$ of $G$ in the sense that
for all $f\in C_{c}(G)$,
\[
\lim_{i\to\infty}\|\pi_{i}(f)\|=\|\pi_{\infty}(f)\|.
\]
Then for any compact subset $K\subset\hat{G}\backslash\mathrm{support}(\pi_{\infty})$,
for $i>i_{0}(K)$, no element of $K$ is weakly contained in $\pi_{i}$.
\end{prop}

\begin{rem}
To attempt to promote Proposition \ref{prop:abstract-strong-convergence-consequence}
to an if and only if statement in the case $\pi_{\infty}=\lambda_{G}$,
one needs to (at least) also discuss possible discrete series representations
of $G$ (for example, in the case $G=\PSL_{2}(\R)$ this issue is
already present). The point is that a diagonal matrix coefficient
of an integrable discrete series\footnote{Integrable meaning having matrix coefficients in $L^{1}$; not all
discrete series have this property.} will act as a non-zero projection in $\pi_{\infty}=\lambda_{G}$
and therefore needs to act in a non-zero way in any $\pi_{i}$ when
$i$ is sufficiently large.
\end{rem}

\begin{proof}
One can show directly from definition of Fell topology that for any
$f\in C_{c}(G)$, the map
\begin{equation}
\pi\mapsto\|\pi(f)\|\label{eq:fmap}
\end{equation}
 is continuous in the Fell topology. For completeness we give this
argument. Given $\pi\in\hat{G}$, suppose that $\xi$ is such that
$\|\xi\|=1$ and 
\[
\langle\pi(f^{*}\ast f)\xi,\xi\rangle=\langle\pi(f)\xi,\pi(f)\xi\rangle>\|\pi(f)\|^{2}-\epsilon.
\]
Essentially by definition of Fell topology (see \cite[Prop. F.2.4]{BdlHV})
there is open set around $\pi$ consisting of $\pi'$ such that 
\[
|\langle\pi'(g)\xi',\xi'\rangle-\langle\pi(g)\xi,\xi\rangle|<\frac{\epsilon}{1+\|f^{*}\ast f\|_{L^{1}}}
\]
 for some $\xi'$ and for all $g$ in the support of $f^{*}\ast f$.
This gives $\|\xi'\|^{2}>1-\epsilon$ (take $g=1)$ and integrating
with $f^{*}\ast f$ weights 
\begin{align*}
\|\pi'(f)\|^{2}(1-\epsilon) & >\|\pi'(f)\|^{2}\|\xi'\|^{2}\geq\langle\pi'(f^{*}\ast f)\xi',\xi'\rangle>\langle\pi(f^{*}\ast f)\xi,\xi\rangle-\epsilon\\
 & >\|\pi(f)\|^{2}-2\epsilon
\end{align*}
for $\pi'$ in this open set. This proves continuity of (\ref{eq:fmap}).

Now let $K$ be as in the statement of the proposition. For every
$\pi\in K$, as $\pi$ is not in the support of $\pi_{\infty}$ there
exists $f_{\pi}\in C_{c}(G)$ and $\eta_{\pi}>1$ such that $\|\pi(f_{\pi})\|>\eta_{\pi}\|\pi_{\infty}(f_{\pi})\|$.
By the previous assertion there is on open neighborhood $W_{\pi}$
of $\pi$ where this inequality still holds. By compactness of $K$
we then obtain a finite list of functions $f_{1},\ldots,f_{r}\in C_{c}(G)$
and $\eta>1$ such that for all $\pi\in K$, 
\begin{align*}
\|\pi(f_{j})\| & >\eta\|\pi_{\infty}(f_{j})\|
\end{align*}
for some $f_{j}$. But for large enough $i>i_{0}$, from strong convergence
\[
\|\pi_{i}(f_{j})\|<\eta\|\pi_{\infty}(f_{j})\|
\]
 for all $j$. Combining the above two inequalities gives for some
$f_{j}$, $\|\pi_{i}(f)\|<\|\pi(f_{j})\|$ so $\pi$ is not in the
support of $\pi_{i}$.
\end{proof}

\section{Applications\label{sec:Applications}}

\subsection{An induction principle}

The following theorem is at the heart of applications of strong convergence
to spectral geometry.
\begin{thm}
\label{thm:induction}Suppose $G$ is locally compact and $\Gamma$
is a cocompact lattice in $G$. If $\rho_{i}\tos\rho_{\infty}$ then
\[
\Ind_{\Gamma}^{G}\rho_{i}\tos\Ind_{\Gamma}^{G}\rho_{\infty}.
\]
\end{thm}

In applications usually one wants $\rho_{\infty}=\lambda_{\Gamma}$
so that $\Ind_{\Gamma}^{G}\rho_{i}\tos\Ind_{\Gamma}^{G}\lambda_{\Gamma}=\lambda_{G}$.
In light of Proposition \ref{prop:abstract-strong-convergence-consequence},
and in the case $G$ is semisimple Lie, the identification of the
unitary dual with spectral parameters, it yields a type of spectral
convergence of $\Ind_{\Gamma}^{G}\rho_{i}$ to the Plancherel measure
of $G$.

The downside of this general argument is that

\textbf{a.} it does not give an effective rate of convergence of spectral
parameters,

\textbf{b.} it does not apply as-is to non-uniform lattices.

Both these issues have been dealt with in special instances (see \cite{MageeThomas}
and \cite{HideMagee} regarding Point \textbf{a) }and \cite{HideMagee}
regarding Point \textbf{b})). 

We now address the proof of Theorem \ref{thm:induction}. It relies
on the following type of `matrix amplification' that is well-known
in the literature e.g. \cite[\S 9]{HaagerupThr}. We take the chance
to record an effective version of this lemma provided by Mikael de
la Salle.
\begin{lem}[Effective matrix amplication]
Let $A$ be a $C^{*}$-algebra and $x\in M_{n}(A)$. For every integer
$p$
\[
\|x\|_{M_{n}(A)}\in[1,n^{\frac{1}{2p}}]\max_{i}\|\left(\left(x^{*}x\right)^{p}\right)_{i,i}\|_{A}^{\frac{1}{2p}}.
\]
\end{lem}

As a result,
\begin{prop}
\label{prop:sc-with-matrix-coefs}If $\Gamma$ is discrete and $\{\rho_{i}\}_{i=1}^{\infty}$
are a sequence of unitary representations of $\Gamma$ with $\rho_{i}\tos\lambda_{\Gamma}$
then for all $r\in\N$ and all $z\in\Mat_{r\times r}(\C)\otimes_{\C}\C[\Gamma]$,
\[
\|[\id\otimes\rho_{i}](z)\|\to\|\id\otimes\lambda_{\Gamma}(z)\|.
\]
\end{prop}

\begin{proof}[Proof of Theorem \ref{thm:induction}]
 Given $f\in C_{c}(G)$, $\pi(f)$ acts on $L^{2}(G,\rho)$ by 
\[
\pi(f)[\phi](h)=\int_{G}f(g)\phi(g^{-1}h)d\mu(g)=\int_{G}f(hg)\phi(g^{-1})d\mu(g).
\]
where $\mu$ is the left invariant Haar measure. This is the same
as 
\begin{align*}
\sum_{\gamma\in\Gamma}\int_{F}f(h\gamma g)\phi(g^{-1}\gamma^{-1})d\mu(g) & =\sum_{\gamma\in\Gamma}\int_{F}f(h\gamma g)\rho(\gamma^{-1})\phi(g^{-1})d\mu(g)\\
 & =\sum_{\gamma}\rho(\gamma^{-1})[A_{f}(\gamma)\phi\lvert_{F}]
\end{align*}
where $A_{f}(\gamma):L^{2}(F,V)\to L^{2}(F,V)$ such that if $\psi\in L^{2}(F,V)$
and $h\in F$, 
\begin{equation}
A_{f}(\gamma)[\psi](h)\eqdf\int_{F}f(h\gamma g)\psi(g^{-1})dg.\label{eq:Af}
\end{equation}
Using $L^{2}(F,V)\cong L^{2}(F)\otimes V$ we obtain a unitary conjugacy
\[
\pi(f)\cong\sum_{\gamma\in\Gamma}a_{f}(\gamma)\otimes\rho(\gamma^{-1})
\]
where $a_{f}(\gamma)$ is defined by the same formula (\ref{eq:Af})
as $A_{f}$ but acting on $L^{2}(F)$.

Now we make two observations. Firstly, $a_{f}(\gamma)=0$ unless there
are $h$ and $g$ in $F$ such that $h\gamma g\in\mathrm{supp}(f)$.
Let $K\subset G$ be compact such that 
\[
G=\bigcup_{\gamma\in\Gamma}\gamma C
\]
and $F\subset C$. The above event is majorized by $\gamma\subset C^{-1}\mathrm{supp}(f)C^{-1}$
which is contained in a compact set. This compact subset can meet
only finitely many elements of $\gamma$. So
\[
\gamma\mapsto a_{f}(\gamma)
\]
 has finite support.

Secondly, $a_{f}(\gamma)$ is an integral operator with kernel $K_{\gamma}(h,g)=f(h\gamma g^{-1})$.
Since $f$ is continuous, and the closure of $F$ is compact, it is
bounded on $F$ and even Hilbert-Schmidt.

So the image of this conjugacy is contained in 
\[
\C[\Gamma]\otimes_{\C}HS(F)
\]
where $HS(F)$ are the Hilbert-Schmidt operators on $L^{2}(F)$ and
hence generate a $C^{*}$-subalgebra of $C_{\red}^{*}(\Gamma)\otimes\KK(L^{2}(F))$
where $\KK$ are the compact operators on a separable Hilbert space.

Now by Proposition \ref{prop:sc-with-matrix-coefs} together with
approximation of compact operators by finite rank ones, we have for
all $f\in C_{c}(G)$ and notation as above, if $\pi_{i}=\Ind_{\Gamma}^{G}\rho_{i}$
and $\rho_{i}\tos\rho_{\infty}$ then

\[
\|\pi_{i}(f)\|=\|\sum_{\gamma\in\Gamma}a_{f}(\gamma)\otimes\rho_{i}(\gamma^{-1})\|\to\|\sum_{\gamma\in\Gamma}a_{f}(\gamma)\otimes\lambda_{\Gamma}(\gamma^{-1})\|\cong\|\Ind_{\Gamma}^{G}\rho_{\infty}\|.
\]
This proves Theorem \ref{thm:induction}.
\end{proof}

\subsection{Hyperbolic surfaces}

In this section $\Gamma$ is the fundamental group of a closed orientable
surface of genus $g\geq2$. Moreover, we discretely embed $\Gamma\hookrightarrow\PSL_{2}(\R)$
in some fixed but arbitrary way, fixing a hyperbolic structure on
a genus $g$ surface $X\eqdf\Gamma\backslash\mathbb{H}$. Importantly,
this embedding could be arithmetic.

By Corollary \ref{cor:surfaces}, there is a sequence $\{\phi_{i}\in\Hom(\Gamma,S_{n_{i}})\}_{i=1}^{\infty}$
such that the induced $\rho_{i}=\std\circ\phi_{i}$ satisfy $\rho_{i}\tos\lambda_{\Gamma}$.
Hence 
\[
\Ind_{\Gamma}^{G}\rho_{i}\tos\lambda_{\PSL_{2}(\R)}
\]
by Theorem \ref{thm:induction}.

Proceeding depends on knowing the unitary dual of $\PSL_{2}(\R)$
and the Plancherel measure. Of interest here are the complementary
series which are outside the support of the Plancherel measure. Now
Proposition \ref{prop:abstract-strong-convergence-consequence} implies
that for any compact subset $K$ of the complementary series, for
$i\gg_{K}1$ no member of $K$ is weakly contained in 
\[
\Ind_{\Gamma}^{G}\rho_{i}.
\]

\emph{\uline{A note.}}\emph{ (On fibered products)}

We now make one more observation. Because $\rho_{i}$ is derived from
$\phi_{i}$, the space of $\Ind_{\Gamma}^{G}\rho_{i}$ is same as
$L^{2}$ sections of the fibered product
\[
\Gamma\backslash_{\phi_{i}}(G\times\ell_{0}^{2}([n_{i}])).
\]
In turn, such sections are the same as $L^{2}$ functions on 
\[
\Gamma\backslash_{\phi_{i}}(G\times[n_{i}])
\]
that have mean zero in every fiber (the above is a covering space
of $\Gamma\backslash G$).

By the relation between the Casimir operator of $\PSL_{2}(\R)$ and
the Laplacian on the hyperbolic surface
\[
X_{\phi_{i}}\eqdf\Gamma\backslash_{\phi_{i}}\left(\mathbb{H}\times[n_{i}]\right)=\Gamma\backslash_{\phi_{i}}(G\times[n_{i}])/P\SO(2),
\]
one obtains as conclusion:
\begin{thm}
As $i\to\infty$, 
\[
\spec(\Delta_{X_{\phi_{i}}})\cap\left[0,\frac{1}{4}-o(1)\right]=\spec(\Delta_{X})\cap\left[0,\frac{1}{4}-o(1)\right].
\]
\end{thm}

By choosing $X$ so that $\Delta_{X}$ has no eigenvalues below $\frac{1}{4}$,
one obtains a sequence of closed hyperbolic surfaces (covering $X$)
with genus tending to $\infty$ and first non-zero eigenvalue tending
to $\frac{1}{4}$. \footnote{In fact, one can arrange so that the original $X$ is arithmetic and
so obtain the above conclusion where all surfaces are arithmetic.
Elaboration on this (also using strong convergence as an essential
ingredient) one can prove that \uline{every} $x\in[0,\frac{1}{4}]$
is a limit point of $\lambda_{1}$ of arithmetic hyperbolic surfaces
\cite{Magee2024}. }
\begin{rem}
This result, which established a conjecture of Buser \cite{Buser}
was first obtained in joint work of the author with Hide \cite{HideMagee}
by a related method. At this time, we only had $\PPF$ for free groups,
so we worked with non-compact surfaces with free fundamental groups
and compactified at the end of the argument following Buser--Burger--Dodziuk
\cite{BBD}. The problem this introduced was that we did not have
access to Theorem \ref{thm:induction} so we had to make a more involved
argument using the resolvent of the Laplacian and cusp-patching techniques.
This technique also yields the following theorem, taking Bordenave--Collins
as input.
\end{rem}

\begin{thm}[Hide--Magee]
Let $X$ be a finite-area non-compact hyperbolic surfaces so that
$\pi_{1}(X)\cong\free$ for some free group $\free$. Let $\phi_{n}$
now be a uniform random element of $\Hom(\free,S_{n})$. Then $X_{\phi_{n}}$
is a uniform random degree$-n$ covering space of $X$. With probability
tending to one as $n\to\infty$

\[
\spec(\Delta_{X_{\phi_{n}}})\cap\left[0,\frac{1}{4}-o(1)\right]=\spec(\Delta_{X})\cap\left[0,\frac{1}{4}-o(1)\right].
\]
\end{thm}

This theorem forms a part of much recent activity on the spectral
gaps of random hyperbolic surfaces \cite{MN1,Magee2023,MageeNaudPuder,Wu2022,Lipnowski2024,Anantharaman2023,Hide2023,Hide2024a,Anantharaman2024}.

It is an interesting question to what extent the `induction principle'
obtained above in Theorem \ref{thm:induction} can be extended to
general non-cocompact and even infinite covolume lattices in e.g.
reductive groups. In recent work \cite{Calderon2024}, it has been
shown that induction of strong convergence works well in the setting
of conformally compact hyperbolic surfaces (of infinite area), and
even gives resonance free regions --- a phenomenon that cannot be
seen solely in the representation theory of $\Ind_{\Gamma}^{\PSL_{2}(\R)}\rho_{i}$.
These questions should be pursued in future work.

\bibliographystyle{amsalpha}
\bibliography{strong_convergence}

\noindent Michael Magee, \\
Department of Mathematical Sciences,\\
Durham University, \\
Lower Mountjoy, DH1 3LE Durham,\\
United Kingdom\\
\texttt{michael.r.magee@durham.ac.uk}
\end{document}